\documentclass[times,11pt]{elsarticle}

\usepackage{amssymb,amsmath,amsthm}
\usepackage[margin=3.5 cm]{geometry}
\newtheorem{theorem}{Theorem}
\newtheorem{propo}{Proposition}
\newtheorem{lem}{Lemma}
\newtheorem{exam}{Example}
\newtheorem{coro}{Corollary}
\newtheorem{define}{Definition}
\newtheorem{remark}{Remark}
\newcommand{\F}{\mathbb F}

\newcommand{\N}{\mathbb N}

\newcommand{\Ord}{\mathcal O}

\newcommand{\Lcm}{\mathrm {lcm}}
\newcommand{\Gcd}{\mathrm {gcd}}

\newcommand{\Tr}{\mathrm{Tr}}
\newcommand{\Char}{\mathrm{char}}
\newcommand{\ord}{\mathrm{ord}}

\begin{document}

\begin{frontmatter}

\title{Nilpotent linearized polynomials over finite fields and applications}
\author{Lucas Reis}
\ead{lucasreismat@gmail.com}
\address{Departamento de Matem\'{a}tica, Universidade Federal de Minas Gerais, UFMG, Belo Horizonte, MG, 30123-970, Brazil}
\begin{abstract}
Let $q$ be a prime power and $\F_{q^n}$ be the finite field with $q^n$ elements, where $n>1$. We introduce the class of the linearized polynomials $L(x)$ over $\F_{q^n}$ such that $$L^{(t)}(x):=\underbrace{L(L(\cdots(x)\cdots))}_{t \quad\text{times}}\equiv 0\pmod {x^{q^n}-x}$$ for some $t\ge 2$, called \textit{nilpotent linearized polynomials} (NLP's). We discuss the existence and construction of NLP's and, as an application, we show how to construct permutations of $\F_{q^n}$ from these polynomials. For some of those permutations, we can explicitly give the compositional inverse map and the cycle structure. This paper also contains a method for constructing involutions over binary fields with no fixed points, which are useful in block ciphers.
\end{abstract}
\end{frontmatter}

\section{Introduction}
Let $q$ be a prime power and $\F_{q^n}$ be the finite field with $q^n$ elements, where $n>1$. Any map from $\F_{q^n}$ to itself can be represented by a polynomial in $\F_{q^n}[x]$. Conversely, any polynomial in $\F_{q^n}[x]$ induces a map from $\F_{q^n}$ to itself. In this context, the $\F_q-$linear maps of $\F_{q^n}$ corresponds to the so called linearized polynomials $L(x)=\sum_{i=0}^{k} a_ix^{q^i}$, $a_i\in \F_{q^n}$. If a polynomial $f(x)\in \F_{q^n}[x]$ induces a permutation in $\F_{q^n}$ we say that $f(x)$ is a permutation polynomial over $\F_{q^n}$.
For many applications in coding theory \cite{cod} and cryptography \cite{cryp}, it is interesting to find new families of permutation polynomials over finite fields. For instance, in block ciphers, permutations of binary fields are used as S-boxes to build a confusion layer in the encryption process and the inverse of this permutation is used in the decryption process. In order to avoid some problems like limited memory, it is interesting to use involutions of binary fields, i.e., permutation polynomials $f(x)\in \F_{2^n}[x]$ such that $f^{-1}(x)=f(x)$ or, equivalently, $f(f(x))=x$. However, a random permutation in $\F_{2^n}$ has $O(1)$ fixed points, while a random involution has $2^{n/2}+O(1)$ fixed points. Therefore an involution with more than $O(1)$ fixed points can be distinguished from random permutations and so can be attacked. In fact, as it was suggested in \cite{canteaut}, the involutions should have no fixed points. For more information about construction and properties of permutation polynomials, see \cite{pan}.

In this paper we introduce the class of the \textit{nilpotent linearized polynomials} (NLP's), i.e., linearized polynomials $L(x)\in \F_{q^n}[x]$ such that
$$L^{(t)}(x)\equiv 0\pmod {x^{q^n}-x}$$
for some $t\ge 2$, where $L^{(t)}(x)$ denotes the ordinary polynomial composition of $L(x)$ with itself $t$ times.

We study the existence and construction of those polynomials, including explicit examples. We describe a method for constructing permutation and complete permutation polynomials from those nilpotent polynomials and, in some particular cases, we determine the compositional inverse map and describe the cycle structure. This paper also includes explicit examples of involutions over binary fields which have no fixed points.

\section{Existence and properties of NLP's}
Throughout this paper, $\F_{q^n}$ denotes the finite field with $q^n$ elements, where $q$ is a prime power and $n>1$. A polynomial $L(x)\in \F_{q^n}[x]$ is said to be \textit{linearized} if $L(x)=\sum_{i=0}^{k}a_ix^{q^i}$. Notice that if $L(x)$ is linearized, then $L(z+y)=L(z)+L(y)$ and $L(az)=aL(z)$ for any $a\in \F_{q}$ and $y, z\in \F_{q^n}$, hence $L(x)$ induces an $\F_q-$linear map of $\F_{q^n}$. Conversely, if $\{\omega_1, \cdots, \omega_n\}$ is any basis of $\F_{q^n}$ over $\F_q$, then the matrix $D=(\omega_i^{q^{j-1}})_{ij}$ is invertible and then, for any $\F_q-$linear map $M$ of $\F_{q^n}$ we have that $L(x)=\sum_{i=0}^{n-1}a_ix^{q^i}$ is the linearized polynomial representation of $M$, where
$$(a_0, \cdots, a_{n-1})^T=D^{-1}(b_1, \cdots, b_n)^T,$$
$b_i=M(\omega_i)$ and $^T$ denotes the transpose.
This is an one-to-one correspondence between the $\F_q-$linear maps of $\F_{q^n}$ and the linearized polynomials $L(x)=\sum_{i=0}^{n-1}a_ix^{q^i}\in \F_{q^n}[x]$. 

\begin{remark}
If $L(x)=\sum_{i=0}^ka_ix^{q^i}\in \F_{q^n}[x]$ and $k>n-1$, then the $\F_q-$linear map of $\F_{q^n}$ induced by $L(x)$ can be represented by another linearized polynomial of the form $L_0(x)=\sum_{i=0}^{n-1}b_ix^{q^i}\in \F_{q^n}[x]$. In fact, $L_0(x)$ is the reduction of $L(x)$ modulo $x^{q^n}-x$. For this reason we will be mostly interested in the linearized polynomials of the form $\sum_{i=0}^{n-1}a_ix^{q^i}$.\end{remark}

\begin{define} If $t\ge 2$ is an integer, we say that $L(x)\in \F_{q^n}[x]$ is a $t-$\textit{nilpotent linearized polynomial} ($t-$NLP) over $\F_{q^n}$ if $L(x)$ is a linearized polynomial such that $L(x)\not\equiv 0\pmod {x^{q^n}-x}$ and
$$L^{(t)}(x)\equiv 0\pmod {x^{q^n}-x}.$$
In other words, $L^{(t)}(x)$ is the zero function when restricted to $\F_{q^n}$ and $L(a)\ne 0$ for some $a\in\F_{q^n}$.
\end{define}
It follows from definition that $L(x)$ is a $t-$NLP over $\F_{q^n}$ if, and only if, its polynomial reduction modulo $x^{q^n}-x$ is a $t-$NLP over $\F_{q^n}$. Moreover, any $t-$NLP is also a $d-$NLP for every $d>t$.

If $f(x)\in \F_{q^n}[x]$, we denote $Z_f=\{z\in \F_{q^n}| f(z)=0\}$ the set of the roots of $f(x)$ in $\F_{q^n}$ and $V_f=\{f(z); z\in \F_{q^n}\}$ its value set over $\F_{q^n}$. If $L(x)$ is a linearized polynomial over $\F_{q^n}$, then $V_L$ and $ Z_L$ are $\F_q$-vector spaces. In fact $V_L$ and $Z_L$ are, respectively, the image and the kernel of the linear map of $\F_{q^n}$ induced by $L$. 

The following theorem gives a necessary and sufficient condition for the existence of $t-$NLP's over $\F_{q^n}$ with prescribed value set $V$.

\begin{theorem}\label{main}
Let $V\subseteq \F_{q^n}$ be any $\F_q-$vector space. Then there exist an integer $t\ge 2$ and a $t-$NLP over $\F_{q^n}$ such that $V=V_L$ if, and only if, $V\ne \{0\}, \F_{q^n}$.
\end{theorem}

\begin{proof}
Suppose that $t\ge 2$ and $L(x)$ is a $t-$NLP over $\F_{q^n}$ such that $V=V_L$. Notice that $L(a)\ne 0$ for some $a\in \F_{q^n}$, hence $V_L\ne \{0\}$. Since $L^{(t)}(z)=0$ for any $z\in \F_{q^n}$, the $\F_q-$linear map of $\F_{q^n}$ induced by $L(x)$ cannot be an isomorphism, hence $Z_L\ne \{0\}$ and then $V_L\ne \F_{q^n}$. Conversely, suppose that $V \ne \{0\}, \F_{q^n}$ and let $\{\omega_1, \cdots, \omega_k\}$ be any basis of $V$ over $\F_{q}$; clearly $k\ne 0, n$. Let $\omega_{k+1}, \cdots, \omega_n$ be elements of $\F_{q^n}$ such that $\{\omega_1, \cdots, \omega_n\}$ is a basis of $\F_{q^n}$ over $\F_{q}$ and $M$ be the $\F_q-$linear map of $\F_{q^n}$ defined as follows: 
$$M(\omega_i)=\begin{cases}\omega_1 & \text{if}\quad  i=n\\   \omega_{i+1} &\text{if} \quad 1\le i\le k-1 \\ 0  &\text{if} \quad k\le i\le n-1.\end{cases}$$
Since $0<k<n$ and $n>1$, $M$ is well defined and a direct calculation shows that $V_M=V$ and $M^{(k+1)}(\omega_i)=0$ for any $1\le i\le n$. Hence $M^{(k+1)}(z)=0$ for any $z\in \F_{q^n}$. Then $L(x)\in\F_{q^n}[x]$, the linearized polynomial representation of $M$, is a $(k+1)-$NLP over $\F_{q^n}$ and satisfies $V_L=V$.

\end{proof}

As it was noticed at the beginning of this section, for a given basis of $\F_{q^n}$ over $\F_q$, the construction of the linearized polynomial associated to a linear map requires only the calculation of the inverse of a matrix. In this context, the proof of Theorem \ref{main} suggests a computational method for constructing $t-$NLP's with a given value set. However, we can find explicit examples of such polynomials: 

\begin{exam}
Let $\Tr_{L/K}(x)=\displaystyle\sum_{i=0}^{\frac{n}{m}-1}x^{q^{im}}$ be the trace of $L=\F_{q^n}$ over an subfield $K$ of the form $\F_{q^m}$. If $\theta$ is an element of $\F_{q^n}^*$ such that $\Tr_{L/K}(\theta)=0$, then $$L_{\theta}(x)=\theta \cdot \Tr_{L/K}(x)$$ is a $2-$NLP over $\F_{q^n}$ and its value set over $\F_{q^n}$ is given by $\theta \cdot \F_{q^m}$. In particular, if $\frac{n}{m}$ is divisible by $p=\Char(\F_q)$, then $L(x)=\Tr_{L/K}(x)=\displaystyle\sum_{i=0}^{\frac{n}{m}-1}x^{q^{im}}$ is a $2-$NLP over $\F_{q^{n}}$.
\end{exam}

\begin{exam}\label{half} Let $m$ be any positive integer and $n=2m$. If $\alpha$ and $\beta$ are two elements in $\F_{q^n}^{*}$ such that $\alpha^{q^m}+\alpha=0$ and $\beta^{q^m+1}=1$, a direct calculation shows that $$L_{\alpha, \beta}(x)=\alpha\beta x^{q^m}+\alpha x$$ is a $2-$NLP over $\F_{q^n}$. The equations $x^{q^m}+x=0$ and $x^{q^{m}+1}=1$ have, respectively, $q^m-1$ and $q^m+1$ solutions over $\F_{q^n}^*$. Hence there are $q^n-1$ polynomials of the form $L_{\alpha, \beta}(x)$.
\end{exam}

\subsection{NLP's in $\F_{q}[x]$}
Here we give a complete characterization of the $t-$NPL's over $\F_{q^n}$ such that their coefficients lie on the base field, i.e, $L(x)\in \F_q[x]$. 
First we need to recall some concepts of the theory of linearized polynomials which can be found in \cite{LiNi}, Section 3.4.

\begin{define}
If $L_1(x), L_2(x)\in \F_{q^n}[x]$ are linearized polynomials we define their \textit{symbolic product} by 
$$L_1(x)\otimes L_2(x)=L_1(L_2(x)),$$
which also is a linearized polynomial.
\end{define}
 A simple calculation shows that the symbolic product $\otimes$ is associative, distributive with respect to the ordinary addition, but is not commutative. However, if $L_1(x), L_2(x)\in \F_q[x]$ it can be verified that $L_1(x)\otimes L_2(x)=L_2(x)\otimes L_1(x)$. 

\begin{define}
Let $L(x)=\sum_{i=0}^{t} a_i x^{q^i}\in \F_{q^n}[x]$ be a linearized polynomial and $l(x)=\sum_{i=0}^{t} a_i x^{i}$. The polynomials $l(x)$ and $L(x)$ are called $q-$associates of each other. More specifically, $l(x)$ is called the conventional $q-$associate of $L(x)$ and $L(x)$ is called the linearized $q-$associate of $l(x)$.
\end{define}

The following lemma shows an interesting property of the linearized polynomials $L(x)\in \F_q[x]$:

\begin{lem}[\cite{LiNi}, Lemma 3.59]\label{lini}
 Let $L_1(x)$ and $L_2(x)$ be linearized polynomials with conventional $q-$associates $l_1(x)$ and $l_2(x)$, respectively. If the coefficients of $L_1$ and $L_2$ lie on the base field $\F_q$, then the polynomials $l(x)=l_1(x)\cdot l_2(x)$ and $L(x)=L_1(x)\otimes L_2(x)$ are $q-$associates. 
\end{lem}

Using Lemma \ref{lini} and the following proposition we will give necessary and sufficient conditions in order for $L(x)$ to be a $t-$NLP over $\F_{q^n}$ in the case when $L(x)\in \F_{q}[x]$.

\begin{propo}\label{comp}
Suppose that the polynomial $g(x)\in \F_{q}[x]$ satisfies $g(x+a)=g(x)$ for every $a\in \F_{q^n}$. Then there exists a polynomial $R(x)\in \F_{q}[x]$ such that $g(x)=R(x^{q^{n}}-x)$. In particular, if $g(x)$ is linearized, then so is $R(x)$.
\end{propo}

\begin{proof}
For the first statement, we proceed by induction on $n=\deg g(x)$. If $g(x)$ is constant then there is nothing to prove. Suppose that the statement is true for all polynomials of degree at most $k$ and let $g(x)\in \F_{q}[x]$ a polynomial of degree $k+1$ satisfying $g(x+a)=g(x)$ for every $a\in \F_{q^n}$. We have $g(0)=g(a)$ for all $a\in \F_{q^n}$ and so the polynomial $g(x)-g(0)$ has degree $k+1>0$ and vanishes at $\F_{q^n}$. In particular we have that 
\begin{align}\label{um} g(x)-g(0)=(x^{q^n}-x) G(x)\end{align}
 for some non-zero polynomial $G(x)\in \F_{q}[x]$. Since $g(x+a)-g(0)=g(x)-g(0)$ and $(x+a)^{q^n}-(x+a)=x^{q^n}-x$ for every $a\in \F_{q^n}$, it follows from \eqref{um} that $G(x)=G(x+a)$  for every $a\in \F_{q^n}$ and $\deg G(x)< \deg g(x)$. By the induction hypothesis we have that $G(x)=F(x^{q^n}-x)$ for some $F(x)\in \F_{q}[x]$. Therefore $g(x)=(x^{q^n}-x)F(x^{q^n}-x)+g(0)$ and so $g(x)=R(x^{q^n}-x)$, where $R(x)=xF(x)+g(0)\in \F_{q}[x]$.

For the second statement, notice that if $g(x)$ is linearized, then the equality $g(x)=R(x^{q^n}-x)$ yields:
\begin{equation}\label{three} bR(z^{q^n}-z)=bg(z)=g(bz)=R((bz)^{q^n}-bz)=R(b(z^{q^n}-z))\end{equation}
\begin{equation}\label{four} R(z^{q^n}-z)+R(y^{q^n}-y)=g(z)+g(y)=g(z+y)=R(z^{q^n}-z+(y^{q^n}-y))\end{equation}
for any $b\in \F_{q}$ and $y, z\in \overline{\F_{q^n}}$, where $\overline{\F_{q^n}}$ denotes the algebraic closure of $\F_{q^n}$. In particular, for any $A, B\in \overline{\F_{q^n}}$ there exist $A_0$ and $B_0$ in $\overline{\F_{q^n}}$ such that $A_0^{q^n}-A_0=A$ and $B_0^{q^n}-B_0=B$ and then, from the equalities \eqref{three} and \eqref{four}, we conclude that 
$$R(A+B)=R(A)+R(B)\quad \text{and}\quad R(bA)=bR(A)$$ 
for any $b\in \F_{q}$ and $A, B\in \overline{\F_{q^n}}$. Hence $R(x)$ induces an $\F_{q}-$linear map $T$ from $\overline{\F_{q^n}}$ to itself. Let $r=\deg R(x)$ and $s$ large enough such that $q^s>r$. If $S(x)$ is the linearized polynomial representation of $T$ when restricted to $\F_{q^s}$, then $R(z)=S(z)$ for any $z\in \F_{q^s}$ and $\deg R, \deg S<q^s$. Therefore $R(x)=S(x)$ and thus $R(x)$ is linearized.

\end{proof}

\noindent The main result of this section is the following:

\begin{theorem}\label{MAIN}
Let $L(x)=\sum_{i=0}^{n-1}a_ix^{q^i}\in \F_{q}[x]$ be a nonzero linearized polynomial and $l(x)=\sum_{i=0}^{n-1}a_ix^{i}$ its conventional $q-$associate. Then $L(x)$ is a $t-$NLP over $\F_{q^n}$ if, and only if, $x^n-1$ divides $l(x)^t$. In particular, if $n$ is not divisible by $p$ then for any $t\ge 2$, there are no $t-$NLP's over $\F_{q^n}$ with coefficients in $\F_q$.
\end{theorem}

\begin{proof}
Suppose that $l(x)^t=(x^n-1)\cdot g(x)$ for some $g(x)\in \F_q[x]$ and $t\ge 2$. From Lemma \ref{lini}, we have that $$L^{(t)}(x)=\underbrace{L(x)\otimes \cdots\otimes L(x)}_{t}=(x^{q^n}-1)\otimes G(x),$$ where $G(x)$ is the linearized $q-$associate to $g(x)$. In particular $G(x)$ is linearized and then $(x^{q^n}-x) \otimes G(x)$ is divisible by $x^{q^n}-x$ in the ordinary sense. Therefore $L^{(t)}(x)\equiv 0\pmod {x^{q^n}-x}$ and, since $\deg L(x)<q^n$ and $L(x)$ is nonzero, we conclude that $L(x)\not\equiv 0\pmod {x^{q^n}-x}$. Thus $L(x)$ is a $t-$NLP over $\F_{q^n}$.

Conversely, suppose that $L(x)$ is a $t-$NLP over $\F_{q^n}$ and set $M(x)=L^{(t)}(x)$. Since $M(x)\in\F_q[x]$ is linearized and vanishes on $\F_{q^n}$, it follows that $M(x+a)=M(x)+M(a)=M(x)$ for any $a\in \F_{q^n}$. From Proposition \ref{comp} there exists a linearized polynomial $R(x)\in \F_{q}[x]$ such that $M(x)=R(x^{q^n}-x)$, i.e, $M(x)=R(x)\otimes (x^{q^n}-x)$. Therefore
\begin{equation}\label{two}L^{(t)}(x)=\underbrace{L(x)\otimes \cdots\otimes L(x)}_{t}= R(x)\otimes(x^{q^n}-x).\end{equation}
Since $R(x)\in \F_{q}[x]$, from Lemma \ref{lini} and equation \eqref{two} we conclude that $$l(x)^t=(x^n-1) r(x),$$ where $r(x)$ is the conventional $q-$associate of $R(x)$. Thus $l(x)^t$ is divisible by $x^n-1$. 

Suppose that $n$ is not divisible by $p$ and there exist $t\ge 2$ and $L(x)$ such that $L(x)\in \F_q[x]$ is a $t-$NLP over $\F_{q^n}$. In particular, if $L_0(x)=\sum_{i=0}^{n-1}b_ix^{i}$ is the reduction of $L(x)$ modulo $x^{q^n}-x$, then $L_0(x)\in \F_{q}[x]$ is also a $t-$NLP over $\F_{q^n}$ and so $l_0(x)=\sum_{i=0}^{n-1}b_ix^i$, the conventional $q-$associate of $L_0(x)$, is such that $l_0(x)^t$ is divisible by $x^n-1$. But if $n$ is not divisible by $p$, then $x^n-1$ has only simple roots and so we conclude that $l_0(x)$ is also divisible by $x^n-1$. Since $l_0(x)$ has degree at most $n-1$ it follows that $l_0(x)=0$, hence $L_0(x)=0$ and so $L(x)\equiv 0\pmod {x^{q^n}-x}$, a contradiction.
\end{proof}

Theorem \ref{MAIN} suggests a method for the construction of $t-$NLP's over $\F_{q^n}$ in the case when $n$ is divisible by $p$:

\begin{coro}\label{const}
Let $t\ge 2$ be an integer, $p=\Char(\F_{q^n})$ and $n=p^su$, where $\gcd(u, p)=1$ and $s\ge 1$. Let $r(x)\in \F_q[x]$ be any nonzero polynomial of degree at most $v=n-1-u\cdot \left \lceil \frac{p^s}{t}\right\rceil$ and 
$$l_{r, t}(x)=r(x)(x^{u}-1)^{\left\lceil\frac{p^s}{t}\right\rceil}.$$ 
Then $L_{r, t}(x)\in \F_{q}[x]$, the linearized $q-$associate of $l_{r, t}(x)$, is a $t-$NLP over $\F_{q^n}$.
\end{coro}

\begin{proof}
A direct calculation shows that $l_r(x)^t$ is divisible by $x^n-1$ and $\deg l_r(x)<n-1$. The result follows from Theorem \ref{MAIN}.
\end{proof}

A simple investigation shows that $v=n-1-u\cdot \left \lceil \frac{p^s}{t}\right\rceil\ge 0$ if $n=p^su$ and $t\ge 2$. The following example is a particular case of Corollary \ref{const} when $r(x)=1$:
\begin{exam}
Let $p=\Char(\F_{q^n})$, $\alpha\in\F_q^*$ and $n=p^su$, where $\gcd(u, p)=1$ and $s\ge 1$. The polynomial $$L_{1, t}(x)=\sum_{i=0}^{d_{p, t}}(-1)^{d_{p,t}-i}\binom{d_{p, t}}{i} x^{q^{ui}}$$ is a $t-$NLP over $\F_{q^n}$, where $d_{p, t}=\left\lceil \frac{p^s}{t}\right\rceil$.
\end{exam}

\section{Constructing permutations via $t-$NLP's}
In this section we present a method for constructing permutation polynomials over $\F_{q^n}$ which are the sum of two polynomials, one of them being a $t-$NLP. A polynomial $f(x)\in \F_{q^n}[x]$ is said to be a permutation polynomial over $\F_{q^n}$ if the map $c\mapsto f(c)$ induced by $f(x)$ is a permutation from $\F_{q^n}$ to itself. We say that a permutation polynomial $f(x)$ is a complete permutation polynomial if $f(x)+x$ is also a permutation polynomial. The set $G(q^n)$ of the permutation polynomials over $\F_{q^n}$ is a group under the polynomial compostiton modulo $x^{q^n}-x$, and this group is isomorphic to the symmetric group $S_{q^n}$. The identity element of $(G(q^n), \circ)$ is the identity map $g(x)=x$ and, for each $f\in G(q^n)$, $\Ord(f)$ denotes the order of $f$ in the group $(G(q^n), \circ)$, i.e, $\Ord(f)=\min\{d>0 | f^{(d)}(z)=z, \forall z\in \F_{q^n}\}$.

The following theorem gives an interesting relation between the $t-$NLP's and some permutation and complete permutation polynomials:

\begin{theorem}\label{PP}
Let $p=\Char(\F_{q^n})$. Let $L(x)$ be a $t-$NLP over $\F_{q^n}$ and $k(x)$ be any linearized permutation polynomial over $\F_{q^n}$ such that, under the ordinary polynomial composition, $k$ commutes with $L$, i.e., $$k\circ L(x)=L\circ k(x).$$ If $s=\Ord(k)$, then

\begin{enumerate}[a)]
\item $L(x)+k(x)$ is also a permutation polynomial over $\F_{q^n}$ and its compositional inverse map is given by $$(L+k)^{(-1)}(x)=\sum_{i=0}^{t-1}(-1)^iL^{(i)}(k^{(s-1-i)}(x)),$$
where $k^{(0)}(x)=L^{(0)}(x)=x$ and $(s-1-i)$ is taken modulo $s$.

\item if $k(x)$ is a complete permutation polynomial over $\F_{q^n}$, then so is $L(x)+k(x)$.

\item $\Ord(L+k)$ divides $\Lcm(s, p^e)$, where $e=\lceil \log_{p} t\rceil$.

\item if $t=2$ and $\Gcd(s, p)=1$, then $\Ord(L+k)=ps$.
\end{enumerate}
\end{theorem}

\begin{proof}
In the proof of this result and many others in this section we use the following identity:
$$(L+k)^{(p^l)}(z)=L^{(p^l)}(z)+k^{(p^l)}(z)$$
for any $z\in \F_{q^n}$ and $l\in \N$, which is the Frobenius identity in the case when $L(x)$ and $k(x)$ are commuting linearized polynomials over $\F_{q^n}$.

a) Notice that: $$(L+k)\circ \left[\sum_{i=0}^{t-1}(-1)^iL^{(i)}(k^{(s-1-i)}(z))\right]=\sum_{i=0}^{t-1}(-1)^iL^{(i+1)}(k^{(s-1-i)}(z))+\sum_{i=0}^{t-1}(-1)^iL^{(i)}(k^{(s-i)}(z))=$$
$$=(-1)^{t-1}L^{(t)}(z)+k^{(s)}(z)=0+z=z,$$
for all $z\in \F_{q^n}$. In particular $L+k$ is a permutation over $\F_{q^n}$ and its inverse map is given by $\sum_{i=0}^{t-1}(-1)^iL^{(i)}(k^{(s-1-i)}(x))$.

b) If $k(x)$ is a complete permutation polynomial, then $K(x)=k(x)+x$ is a permutation polynomial and item (a) shows that $L(x)+k(x)$ is a permutation polynomial. A direct calculation shows that $K\circ L(x)=L\circ K(x)$ and then (a) shows thatr $K(x)+L(x)=(L(x)+k(x))+x$ is a permutation polynomial. Thus $L(x)+k(x)$ is a complete permutation polynomial.

c) Let $v=\Ord(L+k)$ and $u=\Lcm(s, p^e)$, where $e=\lceil \log_pt\rceil$ satisfies $p^e\ge t$. 
In particular, $L^{(p^e)}(z)=0$ for any $z\in\F_{q^n}$.
Since $p=\Char(\F_{q^n})$ and $u$ is divisible by $p^e$ and $s$, then the following equality holds for any $z\in \F_{q^n}$: $$(L+k)^{(u)}(z)=((L+k)^{(p^e)})^{(u/p^e)}(z)=(L^{(p^e)}+k^{(p^e)})^{(u/p^e)}(z)=k^{(u)}(z)=z.$$
Thus $\Ord(L+k)=v$ divides $u$.  

d) Suppose that $t=2$ and $\Gcd(s, p)=1$. In particular $e=\lceil \log_p t\rceil=1$ and item (c) shows that $v=\Ord(L+k)$ divides $u=\Lcm(s, p)=ps$. Since $L^{(2)}(z)=0$, for any $z\in \F_{q^n}$ and $d\in \N$ we have the following equality:
$$(L+k)^{(d)}(z)=dL(k^{(d-1)}(z))+k^{(d)}(z),$$
which is the version of the Binomial Theorem in the in the case when $L(x)$ an $k(x)$ are commuting linearized polynomials over $\F_{q^n}$ and $L^{(2)}(z)=0$ for any $z\in \F_{q^n}$. If $v=\Ord(L+k)$ is not divisible by $p$, then $v$ divides $s$. Therefore $$z=(L+k)^{(s)}(z)= sL(k^{(s-1)}(z))+k^{(s)}(z)=sL(k^{(s-1)}(z))+z$$ or, equivalently, $sL(k^{(s-1)}(z))=0$ for all $z\in \F_{q^n}$. Since $k^{(s-1)}(z)$ is a permutation polynomial over $\F_{q^n}$ and $s$ is not divisible by $p$, it follows that $L(z)=0$ for any $z\in \F_{q^n}$, a contradiction with $L(x)\not\equiv 0\pmod {x^{q^n}-x}$ . Thus $p$ divides $v$ and so there exists some divisor $s_0$ of $s$ such that $v=ps_0$. Therefore, for any $z\in \F_{q^n}$ we have:
$$z=(L+k)^{(ps_0)}(z)=(L^{(p)}+k^{(p)})^{(s_0)}(z)=k^{(ps_0)}(z).$$
Since the equality above holds for all $z\in \F_{q^n}$, it follows that $s=\Ord(k)$ divides $ps_0=v$. Thus $v$ is divisible by $u=\Lcm(s, p)=ps$ and, since $v$ divides $u=ps$ we conclude that $v=u=ps$.
\end{proof}
In a particular case when $k(x)=\gamma x$, where $\gamma\in \F_{q}^*$ we have the following:
\begin{coro}\label{linear}
Let $L(x)$ be a $t-$NLP over $\F_{q^n}$, $p=\Char(\F_{q^n})$ and $\gamma$ be any element of order $s=\ord\gamma$ in the multiplicative group $\F_q^{*}$. Then
\begin{enumerate}[a)]
\item $L(x)+\gamma x$ is a permutation polynomial over $\F_{q^n}$ and its compositional inverse map is given by $$\sum_{i=0}^{t-1}\gamma^{s-1-i}(-1)^iL^{(i)}(x).$$
\item if $\gamma\ne -1$, then $L(x)+\gamma x$ is also a complete permutation polynomial.
\item $\Ord(L+\gamma x)$ divides $p^e\cdot s$, where $e=\lceil \log_pt\rceil$. Also, if $t=2$, then $\Ord(L+\gamma x)=ps$.
\end{enumerate} 
\end{coro}
\begin{proof}
Since $\gamma \in \F_q^{*}$, it follows that $\gamma x$ is a permutation polynomial over $\F_{q^n}$, commutes with $L(x)$ and satisfies $\Ord(\gamma x)=\ord \gamma=s$. Also if $\gamma \ne -1$, $\gamma x$ is a complete permutation polynomial. Finally, since $\ord \gamma=s$ divides $q-1$, we have that $\Gcd(s, p)=1$ and $\Lcm (s, p^d)= p^d\cdot s$ for any $d\in \N$. The results now follow directly from Theorem \ref{PP}.
\end{proof}
\begin{exam}
Let $n=2m$, $\alpha$ and $\beta$ be elements of $\F_{q^n}^{*}$ such that $\alpha^{q^m}+\alpha=0$ and $\beta^{q^m+1}=1$ and $\gamma$ be any element of $\F_q^{*}$. From Example \ref{half} and Corollary \ref{linear}, the polynomials $$L_{\alpha, \beta, \gamma}(x)=(\alpha\beta x^{q^m}+\alpha x)+\gamma x$$ are permutation polynomials over $\F_{q^n}$, $\Ord(L_{\alpha, \beta, \gamma})=p \cdot \ord \gamma$ and the compositional inverse map of $L_{\alpha, \beta, \gamma}(x)$ is given by:
$$\gamma^{-1}x-\gamma^{-2}(\alpha\beta x^{q^m}+\alpha x).$$
\end{exam}

From Corollary \ref{const}, we can construct a large class of permutations:

\begin{coro}\label{general}
Let $t\ge 2$ be an integer, $p=\Char(\F_{q^n})$ and $n=p^su$, where $\gcd(u, p)=1$ and $s\ge 1$. Let $r(x)\in \F_q[x]$ be any nonzero polynomial of degree at most $v=n-1-u\cdot \left \lceil \frac{p^s}{t}\right\rceil$ and 
$$l_r(x)=r(x)(x^{u}-1)^{\left\lceil\frac{p^s}{t}\right\rceil}.$$ 
Also, let $L_r(x)\in \F_{q}[x]$ be the linearized $q-$associate of $l_r(x)$ and $\alpha, \beta$ be elements of  $\F_{q}^*$. Then the polynomials
$$L_{r, \alpha, \beta}(x)=L_r(x)+\alpha \Tr(x)+\beta x$$
are permutation polynomials over $\F_{q^n}$, where $\Tr(x)=\sum_{i=0}^n x^{q^i}$ denotes the absolute trace. Moreover if $\beta\ne -1$, then $L_{r, \alpha, \beta}(x)$ is a complete permutation polynomial over $\F_{q^n}$.
\end{coro}

\begin{proof}
 Since $n$ is divisible by $p=\Char(\F_{q^n})$ and $\alpha\in\F_{q}^*$, a direct calculation shows that the polynomial $\alpha \Tr(x)$ is a $2-$NLP over $\F_{q^n}$. From Corollary \ref{linear}, $\alpha\Tr(x)+\beta x$ is a permutation over $\F_{q^n}$ and $\alpha\Tr(x)+\beta x$ is also a complete permutation polynomial in the case when $\beta\ne -1$. From Corollary \ref{const}, $L_r(x)$ is a $t-$NLP over $\F_{q^n}$. But $L_r(x)$ and $\alpha \Tr(x)+\beta x$ belong to $\F_q[x]$ and so these polynomials commute with each other. Now we apply Theorem \ref{PP} to $L(x)=L_r(x)$ and $k(x)=\alpha \Tr(x)+\beta x$.
\end{proof}

In the notation of Corollary \ref{general}, we give explicit examples of permutation polynomials over $\F_{2^6}$ and $\F_{3^{3}}$ of the type $	L_{r, \alpha, \beta}(x)$:
\\

\def\arraystretch{1.2}
\begin{tabular}{|c|c|} 
\hline
$q=t=2$, $n=6$ & $\alpha=\beta=1$
\\

$r(x)$ & $L_{r, 1, 1}(x)$ \\
\hline

$1$ & $x^{32}+x^{16}+x^4+x^2+x$ \\

$x$ & $x^{32}+x^8+x^4$ \\

$x+1$ & $x^{32}+x^4+x$ \\

$x^2$ & $x^{16}+x^8+x^2$ \\

$x^2+1$ & $x^{16}+x^2+x$ \\

$x^2+x$ & $x^8$\\

$x^2+x+1$ & $x$\\
\hline
\end{tabular}
\quad
\begin{tabular}{|c|c|} 
\hline
$q=n=t=3$ & $\alpha=1, \beta=-1$
\\
$r(x)$ & $L_{r, 1, -1}(x)$ \\
\hline

$1$ & $x^9-x^3-x$ \\

$-1$ & $x^9+x$ \\

$x$ & $-x^9$ \\

$-x$ & $-x^3$ \\

$x+1$ & $-x^9+x^3-x$ \\

$x-1$ & $-x^9-x^3+x$\\

$-x+1$ & $x^3+x$\\

$-x-1$ & $-x$\\
\hline
\end{tabular}

\subsection{Cycle Structure}
If $F$ is any function from a finite set $S$ to itself, we can associate to it a directed graph $G(F, S)$ with vertex set $S$ and edge set $\{(x, F(x))\}_{x\in S}$. We say that $G(F, S)$ is the \textit{functional graph} associated to $F$. If $f(x)$ is a permutation polynomial over $\F_{q^n}$, it can be verified that the graph $G_f := G(f, \F_{q^n})$ is decomposed into disjoint cycles. Moreover, $\Ord(f)$ is the least common multiple of the cycle lengths of $G_f$ and the vertex of $G_f$ associated to $a\in \F_{q^n}$ belongs to a cycle of length $d$ if, and only if, $d$ is the least positive integer such that $f^{(d)}(a)=a$.

If $L(x)$ and $k(x)$ are linearized polynomials over $\F_{q^n}$ as in Theorem \ref{PP}, we know that $L(x)+k(x)$ is a permutation polynomial over $\F_{q^n}$. What is the relation between the functional graphs $G_k$ and $G_{L+k}$?

In the case when $L(x)$ is a $2-$NLP over $\F_{q^n}$, the following theorem shows that the cycle lengths of $G_{L+k}$ cannot be much larger than the ones of $G_{k}$ and, imposing an additional condition on $\Ord(k)$, we can completely describe the cycle structure of $G_{L+k}$ from $G_k$.

\begin{theorem}\label{cycle}
Let be a $2-$NLP over $\F_{q^n}$ and let $k(x)$ be any linearized permutation polynomial over $\F_{q^n}$ such that, under the ordinary polynomial composition, $k(x)$ commutes with $L(x)$, i.e., $k\circ L(x)=L\circ k(x)$. Set $s=\Ord(k)$ and $p=\Char(\F_{q^n})$. Suppose that the vertex associated to an element $a\in \F_{q^n}$ belongs to cycles of lengths $m_{a}$ and $m_{a}'$ in $G_k$ and $G_{L+k}$, respectively. Then the following holds:
\begin{enumerate}[a)]
\item $m_{a}'$ divides $\Lcm(m_{a}, p)$ and, if $L(a)=0$, then $m_{a}=m_{a}'$.
\item if $\Gcd(s, p)=1$ then $m_{a}'=\begin{cases}m_{a} \quad \text{if} & L(a)=0 \\ pm_{a}  & \text{otherwise.}\end{cases}$
\end{enumerate}
\end{theorem}

\begin{proof}
a) Let $v=\Lcm(m_{a}, p)$. Notice that $(L+k)^{(v)}(z)=(L^{(p)}+k^{(p)})^{(v/p)}(z)=k^{(v)}(z)$ for any $z\in\F_{q^n}$. Since $m_{a}$ divides $v$, it follows that $k^{(v)}(a)=a$, hence $(L+k)^{(v)}(a)=a$ and so $m_{a}'$ divides $v$. If $L^{(2)}(z)=0$, we have seen that $$(L+k)^{(d)}(z)=dk^{(d-1)}(L(z))+k^{(d)}(z),$$ for any $d\in \N$ and $z\in\F_{q^n}$. Therefore, if $L(a)=0$ then $(L+k)^{(d)}(a)=k^{(d)}(a)$. Thus $m_{a}'=m_{a}$ if $L(a)=0$.

b) If $L(a)=0$, item (a) shows that $m_{a}=m_{a}'$. Suppose that $L(a)\ne 0$ and that $m_{a}'$ is not divisible by $p$. It follows from item (a) that $m_{a}'$ divides $m_{a}$ and then
$$a= (L+k)^{(m_{a})}(a)=m_{a}L(k^{(m_{a}-1)}(a))+k^{m_{a}}(a)=m_{a}k^{(m_{a}-1)}(L(a))+a,$$
hence $m_{a}k^{(m_{a}-1)}(L(a))=0$. Now, since $\Gcd(s, p)=1$ and $m_{a}$ divides $s$, it follows that $m_{a}$ is not divisible by $p$, hence $k^{(m_{a}-1)}(L(a))=0$. Notice that $k^{(m_{a}-1)}(x)$ is a linearized permutation polynomial and then maps the zero element to itself. Since $L(a)\ne 0$, it follows that the composition $k^{(m_{a}-1)}(L(a))$ is never zero and so we get a contradiction. Thus $p$ divides $m_{a}'$ and then there exists an integer $u$ such that $m_{a}'=pu$. Therefore
$$a=(L+k)^{(pu)}(a)=(L^{(p)}+k^{(p)})^{(u)}(a)=k^{(pu)}(a),$$
and then $m_{a}$ divides $pu$. Since $m_{a}$ is not divisible by $p$ it follows that $m_{a}$ divides $u$, hence $pm_{a}$ divides $pu=m_{a}'$. Item (a) shows that $m_{a}'$ divides $\Lcm(m_{a}, p)=pm_{a}$ and thus $m_{a}'=pm_{a}$.
\end{proof}

In the case when $k(x)=\gamma x$ for some $\gamma \in \F_{q}^*$, we can determine precisely the graphs $G_{L+k}$:

\begin{coro}
Let $L(x)$ be a $2-$NLP over $\F_{q^n}$ and $\gamma$ be an element of order $s$ in the multiplicative group $\F_{q}^*$. Then the functional graph $G_{L+\gamma x}$ has one cycle of length $1$, $\displaystyle \frac{z_L-1}{s}$ cycles of length $s$ and $\displaystyle\frac{q^n-z_L}{ps}$ cycles of length $ps$, where $z_L=\# Z_L$ is the number of roots of $L$ in $\F_{q^n}$. In particular, if $L_1$ and $L_2$ are $2-$NLP's over $\F_{q^n}$ and $\gamma_1, \gamma_2\in \F_q^*$, then the graphs $G_{L_1+\gamma_1 x}$ and $G_{L_2+\gamma_2 x}$ have the same cycle structure (hence isomorphic) if, and only if,  $z_{L_1}=z_{L_2}$ and $\ord \gamma_1=\ord \gamma_2$.
\end{coro}

\begin{proof}
For the first statement, notice that any nonzero element belongs to a cycle of length $s$ in $G_{\gamma x}$ and the zero element is a fixed point. Since $\Ord(\gamma x)=s$ and $s$ divides $q-1$, we have that $\Gcd(d, p)=1$ and now the result follows from part b) of Theorem \ref{cycle}. The second statement follows directly from the first.

\end{proof}

\subsection{Involutions in binary fields}
Here we are interested in the construction of involutions over binary fields with no fixed points. Let $q$ be a power of $2$ and $L(x)$ be any $2-$NLP over $\F_{q^n}$. From Theorem \ref{PP} we know that $L(x)+x$ is a permutation polynomial over $\F_{q^n}$ and it can be verified that $L(x)+x$ is in fact an involution over $\F_{q^n}$. However, $L(x)+x$ has many fixed points which are exactly the roots of $L(x)$ over $\F_{q^n}$. The following proposition shows how to completely eliminate those fixed points:

\begin{propo}\label {invo}
Let $\F_{q^n}$ be a finite field such that $\Char(\F_{q^n})=2$ and $L(x)$ be any $2-$NLP over $\F_{q^n}$ such that $V_L\subsetneq Z_L$. Then, for any $a\in Z_L\setminus V_L$, the polynomial $f(x)=L(x)+x+a$ is an involution over $\F_{q^n}$ with no fixed points. In particular, if $\dim_{\F_q} V_L< n/2$ then there is some element $b\in Z_L\setminus V_L$.
\end{propo}

\begin{proof}
Since $L(a)=0$ and $\Char(\F_{q^n})=2$, a direct calculation shows that $f(x)=L(x)+x+a$ is an involution over $\F_{q^n}$. If $f(x)$ has a fixed point $\alpha\in \F_{q^n}$, then $f(\alpha)=\alpha$ and so $L(\alpha)=a$, which is impossible since $a\not\in V_L$. Thus $f(x)$ has no fixed points. 

Since $L^{(2)}(z)=0$ for any $z\in \F_{q^n}$ we have that $V_L\subset Z_L$. If $\dim_{\F_q} V_L< n/2$ then $\dim_{\F_q} Z_L= n-\dim_{\F_q} V_L>n/2$ and so $V_L\subsetneq Z_L$. Thus there is some element $b\in Z_{L}\setminus V_L$.
\end{proof}

In particular we have the following:

\begin{coro}\label{inv}
Let $q$ be a power of 2 and let $k=\F_{q^m}$ and $K=\F_{q^n}$ be fields such that $k\subset K$ and $m<n/2$. If $\theta$ is an element of $\F_{q^n}^{*}$ and $\Tr_{K/k}(\theta)=0$, then there exists an element $\alpha\in \F_{q^n}$ such that $\Tr_{K/k}(\alpha)=0$ and $\alpha\not\in \theta\cdot \F_{q^m}$. In particular, 
$$f(x)=\theta\cdot \Tr_{K/k}(x)+x+\alpha$$ 
is a involution over $\F_{q^n}$ with no fixed points.
\end{coro}

\begin{proof}
In the notation of Proposition \ref{invo}, take $L(x)=\theta \cdot \Tr_{K/k}(x)$ and notice that $V_{L}=\theta \cdot \F_{q^m}$ has dimension $m<n/2$ as an $\F_q-$vector space. Now the result follows directly from Proposition \ref{invo}.
\end{proof}

The corollary above suggests explicit constructions of involutions with no fixed points which can be represented by \textit{sparse polynomials}, i.e., polynomials with few nonzero coefficients. For instance, let $m$ be any positive integer and $n=4m$. If $f(x)$ is any irreducible polynomial over $\F_{2}$ and has degree $n$, then $\F_{2^n}=\F_2[x]/(f(x))=\F_2[\beta]$, where $\beta$ is the coset of $x$ in the quotient $\F_2[x]/(f(x))$. Take $K=\F_{2^n}$ and $k=\F_{2^m}$, hence
$$\Tr_{K/k}(x)=x^{2^{3m}}+x^{2^{2m}}+x^{2^m}+x.$$ 
A direct calculation shows that $\Tr_{K/k}(1)=\Tr_{K/k}(\beta^{2^m}+\beta)=0$. But if $\beta^{2^m}+\beta \in \F_{2^m}$ then 
$$\beta^{2^{2m}}+\beta^{2^m}=(\beta^{2^m}+\beta)^{2^m}=\beta^{2^m}+\beta,$$ 
hence $\beta^{2^{2m}}=\beta$, i.e., $\beta\in \F_{2^{2m}}$. Therefore $\F_{2^n}=\F_2[\beta]\subset \F_{2^{2m}}$, a contradiction since $n=4m$. In conclusion, $\beta^{2^m}+\beta\not\in \F_{2^{2m}}$ and then taking $\alpha=\beta^{2^m}+\beta$ and $\theta=1$ as in Corollary \ref{inv} we have that
$$f(x)=x^{2^{3m}}+x^{2^{2m}}+x^{2^m}+\beta^{2^m}+\beta$$
is an involution over $\F_{2^n}=\F_2[\beta]$ with no fixed points.
\begin{exam}
Let $\F_{2^{32}}=\F_2[x]/(x^{32}+x^7+x^3+x+1)=\F_2[\beta]$, where $\beta$ is the coset of $x$ in the quotient $\F_2[x]/(x^{32}+x^7+x^3+x+1)$. The polynomial
$$f(x)=x^{2^{24}}+x^{2^{16}}+x^{2^8}+\beta^{2^8}+\beta$$
is an involution over $\F_{2^{32}}$ and has no fixed points.

\end{exam}


\end{document}